\documentclass[11pt,a4paper]{amsart}
\usepackage{amsmath}
\usepackage{amssymb}
\usepackage{hyperref}
\usepackage{cleveref}
\usepackage{mathrsfs}
\usepackage{tikz}

\usepackage{graphicx}
\graphicspath{{./figures/}}
\usepackage{a4wide}
\usetikzlibrary{calc}
\usepackage{tikz-cd}
\usepackage{adjustbox}
\usepackage{verbatim}
\usepackage{soul} 

\usepackage[a4paper, margin=3.20cm]{geometry}
\usepackage{xifthen}
\usepackage{pdfpages}
\usepackage{transparent}
\usepackage{import}

\author{Roberto Ladu}
\address{Ruhr-Universität Bochum, Universitätsstrasse 150, 44801 Bochum, Germany}
\email[Roberto Ladu]{roberto.ladu@rub.de}

\author{Simone Tagliente}
\address{Alfréd Rényi Mathematical Research Institute, 1053 Budapest, Réaltanoda Street 13-15, and Eötvös Loránd
University, Pázmány Péter sétány 1/C., H-1117 Budapest, Hungary
}
\email[Simone Tagliente]{tagliente.simone@renyi.hu}

\newtheorem{thm}{Theorem}[section]
\newtheorem{prop}[thm]{Proposition}
\newtheorem{lem}[thm]{Lemma}

\theoremstyle{definition}
\newtheorem{definition}[thm]{Definition}

\theoremstyle{remark}
\newtheorem{remark}[thm]{Remark}

\numberwithin{equation}{section}

\newcommand{\N}{\mathbb{N}}
\newcommand{\Z}{\mathbb{Z}}

\newcommand{\R}{\mathbb{R}}

\renewcommand{\S}{\mathbb{S}}

\newcommand{\CP}{\mathbb{CP}}
\newcommand{\T}{\mathbb{T}}

\newcommand{\D}{\mathbb{D}}

\newcommand{\genus}{\mathbf{g}}

\newcommand{\Spinc}{{\mathrm{Spin}^c}}
\newcommand{\X}{\mathcal{X}}
\newcommand{\BaykurHamada}{B}
\newcommand{\BaykurHamadaSurgered}{\mathcal{B}}

\newcommand{\Szabo}{Szab\'{o} }

\title{On smooth structures over $4$-manifolds with fundamental group of even order}

\begin{document}
\maketitle
\begin{abstract}
    We show that any topological, closed, oriented, non-spin $4$-manifold with fundamental group $\Z_{4k}$ and $\min(b_2^+, b_2^-)\geq 15$, has either none or infinitely many distinct smooth structures. 
    Furthermore, we construct infinitely many non-diffeomorphic, irreducible, smooth structures on manifolds with signature zero, $b_2^+$ even and fundamental group $\Z_2\times G$, for any finite group $G$.
    This extends the results of Baykur-Stipsicz-Szab\'{o} \cite{BaykurStipsiczSzabo}.
\end{abstract}
\section{Introduction}
It is only in dimension four that a topological manifold can carry infinitely many non-diffeomorphic smooth structures
\cite{KirbySiebenmann, Freedman82, donaldson1987irrationality}.  
Whether every smoothable topological $4$-manifold has this property is still unknown. 
A substantial body of work has focused on constructing closed, oriented,  simply-connected examples with decreasing Euler characteristic, aiming at approaching the lower bound of $\S^4$ \cite{FintushelSternDoubleNodeNeigh,ParkStipsiczSzaboExotic,BaldridgeKirkSymplectic,akhmedov2010exotic}. 
Regarding the non simply-connected case, Park constructed examples of large Euler characteristic with any fundamental group \cite{ParkGeographyOfSymplectic}.
Later, smaller examples were found for finite fundamental groups \cite{Topkara,TorresYazinski} and even smaller examples were constructed by Torres
for the fundamental groups $\Z, \Z_p$ or $\Z_p\times \Z_q$  for certain values of $p,q$ \cite{Torres1, Torres2}. All these examples have $b_1 + b_2^+$ odd because have a symplectic representative.
In \cite{LevinePiccirillo2023new}, Levine, Lidman, and Piccirillo constructed the first exotic smooth structure on a negative definite manifold with fundamental group $\Z_2$, in particular their example has even $b_2^+$.

Recently, Baykur, Stipsicz, and Szab\'o established a general result: 
any  topological, closed, smoothable, oriented, non-spin, $4$-manifold with fundamental group $\Z_{4k+2}$ has infinitely many smooth structures with possibly seven exceptions for each $k$ \cite{BaykurStipsiczSzabo}. In this paper we extend their result to cover the fundamental group $\Z_{4k}$ with $k\geq 1$.
\begin{thm}\label{Thm1} Let $a,b\in \N$ with $a,b\geq 15$.  Then the smoothable closed, non-spin, topological, oriented $4$-manifold $Q$ with $\pi_1(Q) \simeq \Z_{4k}$, $b_2^+(Q)=a, b_2^-(Q)=b$, admits  infinitely many  non-diffeomorphic smooth structures.
\end{thm}
Pairing this with 
\cite{BaykurStipsiczSzabo} gives a statement valid for $\pi_1(Q)\simeq \Z_{2k}$. In  \Cref{Thm1}, the cases where  $b_2^+$ or $b_2^-$ is odd follow from a result of \cite{Torres2} using \cite{BaykurHamada} as input.
 Thus, to prove the theorem, it suffices to consider the cases with $b_2^+$ even and signature zero,  all other cases follow by blowing up or changing the orientation.
Requiring  $b_2^+$ even considerably complicates the task. Indeed, it implies that the Seiberg-Witten invariants vanish as the associated moduli spaces have odd dimension.

This paper includes also results about finite, non-cyclic fundamental groups.
\begin{thm}\label{Thm2}  Let $G$ be a finite group. Then there exists a closed,  topological $4$-manifold $Q$ with $\pi_1(Q)\simeq \Z_2\times G$, $\sigma(Q) = 0$ and $b^+(Q)$ even carrying infinitely many non-diffeomorphic, irreducible, smooth structures. 
\end{thm}
The $\Z_2$ factor 
is due to the strategy we use to obtain an even $b_2^+$.
Note that the case  $\pi_1(Q)= \Z_2 \times G$ with $G$  a finite subgroup of $SO(4)$ follows also from \cite[Thm. 1.3]{BaykurStipsiczSzabo}. 

The approach used in the proof of \Cref{Thm1} and \Cref{Thm2} is similar
and is based on the strategy initiated by \cite{LevinePiccirillo2023new} and \cite{BaykurStipsiczSzabo}.
We begin by constructing a smooth manifold $X$ with a free $G$-action. Then, we perform $G$-equivariant surgeries along tori to kill the fundamental group of $X$ while retaining a free $G$-action. 
Performing more equivariant surgeries, we find an infinite family $\{\X_n\}_{n\in \N}$ of simply-connected manifolds with a free $G$-action which are pairwise non-diffeomorphic. 
Finally we construct our examples by taking quotients. 

In \cite{BaykurStipsiczSzabo}, the authors first construct exotic manifolds with fundamental group $\Z_2$, then,  by performing a circle sum operation with $\S^1\times Y^3$,  they obtain exotic manifolds with fundamental group $\Z_2\times G$ with $G = \pi_1(Y)$, $Y$ being a rational homology $3$-sphere. 
Their result on fundamental groups $\Z_{4k+2}$ then follows by taking  $G = \Z_{2k+1}$.

In contrast, our construction has a combinatorial structure that is specific to the fundamental group we want to realize.
This strategy, on one hand, allows us to obtain more fundamental groups but on the other hand, 
 it requires to perform additional surgeries that depend on the specific group under consideration.
 This introduces the extra difficulty 
 of finding a suitable link of surgery tori satisfying several properties and also accounts for the higher Euler characteristic of our manifolds compared to those in \cite{BaykurStipsiczSzabo}.


\subsection{Organization.} In \Cref{sec:Conventions} we discuss our conventions and review some background useful to the next chapters.
In \Cref{sec:ProofZ4} we prove \Cref{Thm1} for the case $k=1$ and in \Cref{sec:GroupZ4k} we explain the modifications needed to prove the general case $k>1$.
In \Cref{sec:Z2G} we prove \Cref{Thm2}.
\subsubsection*{Acknowledgements} 
The first author is thankful to Marco Marengon for inviting him to the R\'enyi institute where this project originated and to \.{I}nan\c{c}  Baykur and Rafael Torres  for bringing to his attention the references \cite{WelschingerLagrangianTori, ParkGeographyOfSymplectic, HoLiLuttingerSurgery}. The second author is thankful to Marco Marengon and  Andr\'as Stipsticz for guidance and encouragement. 
 The first author is grateful to Ruhr-Universit\"{a}t Bochum for financial support. The second author was supported by the grant EXCELLENCE-151337.

\section{Conventions and background}\label{sec:Conventions}
\subsection{Conventions and notation}
We work in the category of smooth, oriented, manifolds with boundary. When we say that two manifolds are different or distinct, we mean not diffeomorphic.
Homeomorphisms, diffeomorphisms and group actions are always assumed to preserve orientation unless stated otherwise.  

An exotic pair is a pair of manifolds $(X_0,X_1)$ which are homeomorphic but not diffeomorphic, we also say that $X_1$ is an exotic copy of $X_0$.
We denote by $\nu S$ a  tubular neighbourhood of a compact submanifold $S$ with compact fiber and by $\nu^\circ S  := \mathrm{int}(\nu S)$ its interior. 
Throughout the paper we will denote by $\Sigma_g$ a model closed Riemann surface of genus $g$. We use the notation $MCG(F,K)$ to denote the mapping class group of $F$ relative to the submanifold $K\subset F$.

For the sake of brevity, we will be  imprecise when talking about the fundamental group, often referring to unbased loops as generators. 
By this we mean that it is possible to fix  a basepoint $\ast$ and for each loop, an  arc joining said loop to $\ast$, so that the induced based loops generate the fundamental group.

We will denote the elements of the cyclic groups $\Z_k$, $k\in \N$
using bold face digits and  additive notation, e.g. $\mathbf{0}, \mathbf{1},\dots, \mathbf{k-1} \in \Z_k$.

\subsection{Remarks on surgery along tori}
 Let $X$ be a connected, compact $4$-manifold and let $T\subset X$ be an embedded $2$-torus. 
Suppose that the normal bundle of $T$ is trivial, $\nu T \simeq T \times \D^2$ and fix a basis $\{\lambda, \beta , [\partial \D^2]\} \in H_1(\partial \nu T)$.

\begin{definition}\label{def:SurgeryAlongToriXp}
We denote by $X_p$ any $4$-manifold obtained by
gluing  $\T^2\times \D^2$ to $ X \setminus \mathrm{int}(\nu T) $
 using a  diffeomorphism 
$\partial (\T^2\times \D^2)\to \partial \nu T$
which sends
\begin{equation}
    [\partial \D^2]\in H_1(\partial (\T^2\times \D^2))\mapsto [\lambda]+p[\partial \D^2] \in H_1(\partial  \nu T).
\end{equation}
\end{definition}
\begin{lem}\label{lem:MeridianTriviality} Suppose that there exists a closed surface of genus $g$, $\Sigma\subset X$ geometrically dual to $T$, with $g$ disjoint curves $y_1,\dots, y_g \subset \Sigma$, whose union does not separate $\Sigma$, 
which are trivial in $\pi_1(X\setminus T)$.
Then the meridians to $T$ are trivial in $\pi_1(X\setminus T)$.
\end{lem}
\begin{proof} Let $\mu =  \Sigma\cap \partial \nu T$ be a circle normal to $T$. Then $\mu$ can be expressed as a product of $g$   commutators (up to conjugation) in $\pi_1(\Sigma)$,   $\mu = c_1[x_1,y_1]c_1^{-1}\dots c_g[x_g,y_g]c_g^{-1}$ where $x_i, c_i \in \pi_1(\Sigma)$ for $i=1,\dots,g$. Therefore, since $y_i$ is trivial in $\pi_1(X)$ 
also $\mu$ is trivial in $\pi_1(X)$. 
Now recall that every two meridians of $T$ are conjugated in $\pi_1(X\setminus T)$. 
\end{proof}
\begin{lem}\label{lem:SurgeryKillingLongitude} Suppose that a meridian of $T$ is trivial in $\pi_1(X\setminus T)$.  Then for any $p\in \N$, 
    $$\pi_1(X_p)\simeq \frac{\pi_1(X)}{N(\lambda)},$$ 
    where $N(\lambda)$ denotes the normal closure of $\lambda$ in $\pi_1(X)$.
\end{lem}
\begin{proof} Choose as generators of $\pi_1(\partial(\T^2\times \D^2)) = \pi_1 (\T^3)$ the three $\S^1$ factors. Then applying Seifert-Van Kampen we obtain:
$$\pi_1(X_p) \simeq \frac{\pi_1(X\setminus T)}{N(\mu^p\lambda)}.$$
Now since $\mu$ is null-homotopic in $\pi_1(X\setminus T)$ we have that $\mu^p\lambda \sim \lambda$ and  $\pi_1(X\setminus T)\simeq \pi_1(X)$ from which the claim follows. \end{proof}

\subsection{Review of the Baykur-Hamada manifold.}\label{sub:BaykurHamadaManifolds} 
Baykur-Hamada defined in \cite{BaykurHamada} a  Lefschetz fibration $\BaykurHamada\to \Sigma_2$  which will be crucial to our construction. We now review $B$, noting the following properties:
\begin{enumerate}
    \item  $B$  has signature zero: $\sigma(B) = 0$,
    \item the generic fiber $F$ has genus $g = 8$ (in \cite{BaykurHamada} $g$ can be taken to be any number $\geq 5$ and we choose $g=8$),
    \item  it has $4$ vanishing circles, 
    \item it has a section $S: \Sigma_2\to \BaykurHamada$ of self intersection $S\cdot S = 0$, consequently $B$ admits a symplectic structure.
    \item $\pi_1(B)$ is generated by loops in the fiber and in the base of the fibration. Specifically, we denote $x_1, y_1$ the loops in the base on the left, $x_2, y_2$ the loops on the right, 
    $a_i, b_i$ for $1\leq i\leq 8$ the loops in the fiber $F$. These form a symplectic base of the first homology of the base and of the fiber, respectively.   The structure of the fibration is described by \Cref{Pic:BaykurHamada}.
    \begin{figure}
        \centering
        \includegraphics[width=0.4\linewidth]{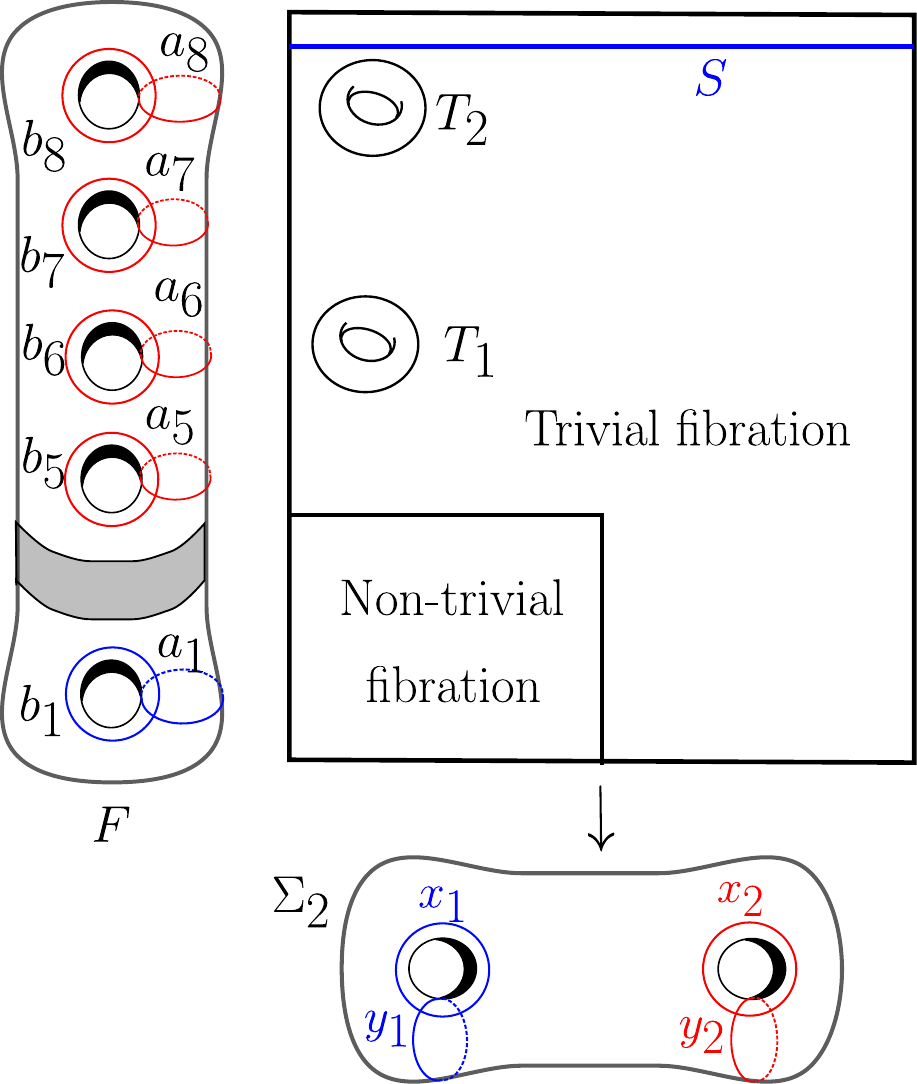}
        \caption{Schematic description of $B\to \Sigma_2$ with generators of $H_1(\Sigma_2)$ and $H_1(F)$, the section $S$ and the tori $T_1, T_2$.}
        \label{Pic:BaykurHamada}
    \end{figure}
        
    \item $B$ is not simply connected, however the fibration possesses certain Lagrangian tori: \begin{align}\label{eq:SurgeryToriBaykurHamada}
        & x_2\times b_1  & & y_2 \times b_2  & & y_2\times a_i  & & y_2 \times b_i \\
        \notag  & x_1 \times a_5 & & y_1 \times a_5 & & b_1\times y_2 &  & b_2\times x_2 
    \end{align}
    where $i=3,\dots 8$,
    such that performing Luttinger surgery on them yields a symplectic and simply-connected manifold $\BaykurHamadaSurgered$.
    More precisely, $\pi_1(\BaykurHamadaSurgered)$ is still generated by the loops on the fiber and the base but now they bound disks. 
    Moreover the section survives in $\BaykurHamadaSurgered$ and has simply-connected complement. 
    Notice that $y_2\times b_2 \cap y_2\times a_2 \neq \emptyset$, but they can be made disjoint by considering a parallel pushoff of $y_2$ on $\Sigma_2$.

    \item There are two extra Lagrangian tori in $\BaykurHamadaSurgered$:
    \begin{align*}
        T_1 := x_1\times a_6  & & T_2 :=x_1 \times a_8
    \end{align*}
     with $\pi_1(\mathcal{B}-(S\cup T_1 \cup T_2))=1$ (for an explicit proof of this last fact, see \cite[Corollary 3.7]{tagliente2026exotic}).
    \item\label{item:Non_SpinBaykurHamada} There is a surface with odd self-intersection lying in the complement of all the surgery tori and the section. In particular $\BaykurHamadaSurgered$ is non-spin.

\end{enumerate}

\section{Manifolds with fundamental group $\Z_{4}$}\label{sec:ProofZ4}
In this section we prove \Cref{Thm1} in the case $k=1$.
\subsection{The manifold $X$} Firstly we will construct a closed symplectic $4$-manifold $X$ with a free $\Z_{4}$-action. 

\newcommand{\ii}{{(i)}}
Consider four disjoint copies of the Baykur-Hamada fibration: $B^\ii := B$, $i=1,\dots, 4$.

We denote the copy of the section $S$ in $B^\ii$ as $S^\ii$
and the copies of the tori $T_1,T_2$ as $T_1^\ii, T_2^\ii\subset B^\ii$.

We construct $X$ as follows. For each $i$ we  remove from $B^\ii$ the interior of the tubular neighbourhoods (the identifications are the same for each $i$)
\begin{align*}
    \nu S^\ii \simeq S^\ii \times \D^2 & &  \nu T_1^\ii \simeq T_1^\ii \times \D^2  & & \nu T_2^\ii \simeq T_2^\ii \times \D^2.
\end{align*}

Then, under the above identifications, we glue together 
\begin{equation}\label{eq:GluingMapSection1}
\begin{split}
\partial \nu S^{(1)}   =  S^{(1)} &\times \S^1 \to S^{(3)} \times \S^1 \simeq \partial \nu S^{(3)} \\
& (x, z)\mapsto (r(x),-z) \\
\end{split}
\end{equation}
where $r:\Sigma_2\to \Sigma_2$ is a free orientation reversing involution preserving the curve $x_1$.
Similarly we glue 
\begin{equation}\label{eq:GluingMapSection2}
\begin{split}
\partial \nu S^{(2)}   =  S^{(2)} &\times \S^1 \to S^{(4)} \times \S^1 \simeq \partial \nu S^{(4)} \\
& (x, z)\mapsto (r(x),-z). 
\end{split}
\end{equation}

Next, we glue 
for $i=1,.., 4$
\begin{equation}\label{eq:GluingMapTori}
\begin{split}
\partial \nu T_1^{(i)}   =  T_1^{(i)} &\times \S^1 \to T_2^{(i+1)} \times \S^1 \simeq \partial \nu T_2^{(i+1)} \\
& (x, z)\mapsto (x,\bar{z}),
\end{split}
\end{equation}
where the superscript is interpreted modulo $ 4$.
We call $X$ the result and we orient $X$ with the orientation induced by the orientation of the blocks $B^\ii$.

\begin{lem}
$X$ admits a symplectic structure.
\end{lem}
\begin{proof} Equip  $B^\ii$ with the symplectic structure given by the Lefschetz fibration for $i = 1,2$, and with the opposite of such symplectic structure for  $i=3,4$. Now the embedding $id: S^\ii\to B^\ii $ for $i=1,2$ and $r:S^\ii\to B^\ii$ for $i=3,4$ are symplectic because $S$ is the section of the Lefschetz fibration.
For any $i$, the tori $T_1^\ii$ and $T_2^\ii$  are Lagrangians in $B^\ii$ because are products of a curve on the fiber and a curve on the base. 

All the surfaces involved in the gluing have trivial normal bundle and the tori have a geometric dual hence are non-trivial in $H_2(B^\ii;\R)$.
Thus we can invoke \cite[Cor. 1.7]{GompfNewConstructions} to give $X$ a symplectic structure.\end{proof}

We define a free action of $\Z_{4}$ on $\bigsqcup_i B^\ii$ by sending each block $B^\ii$ to $B^{(i+1)}$  ciclically for $i=1,\dots, 4$ as the identity. It is not difficult to check that this action induces an action on $X$ since it passes to the quotient by the identifications we used above.
Moreover, since  $r:\Sigma_2\to \Sigma_2$ is a free involution, $\Z_{4}$ acts freely on $X$.
 
\subsection{The manifolds $\X_p$} The manifold $X$ has a free $\Z_{4}$-action but is not simply-connected. We will now construct, by performing surgeries on $X$,  manifolds $\X_p$, which are simply-connected, carry a free $\Z_{4}$-action and have non-trivial Seiberg-Witten invariants.

\subsection{}\label{subsec:ConstructionXhi} We begin by performing on each block $B^\ii$ Luttinger surgeries as done in \cite{BaykurHamada}, more precisely we use in each block the tori \eqref{eq:SurgeryToriBaykurHamada}. 
We denote the resulting manifold as $\X$. 
Since each block is surgered in the same way, the free $\Z_{4}$-action on $X$ induces a free $\Z_{4}$-action on $\X$.
Moreover $\X$ is symplectic because the surgered tori are Lagrangians and we performed Luttinger surgeries \cite{luttinger1995lagrangian}. 

$\X$ can be thought as taking four copies of $\BaykurHamadaSurgered$ and performing  identifications analogous to those of $X$. Hence, since $\pi_1(\BaykurHamadaSurgered) = 1$, non-trivial  loops in $\pi_1(\X)$ are due to blocks' identification.
Indeed, we can take as generators of $\pi_1(\X)$, curves $\gamma_1,\gamma_2, \gamma_3, \gamma_4 \subset X$ where $\gamma_i$ is a curve crossing $B^\ii, B^{(i+1)}$, and  $B^{(i+2)}$, and, in addition, $\gamma_{i+1} = \mathbf{1} \cdot  \gamma_i$ where $\mathbf{1}\in \Z_{4}$ is the generator. Actually $\gamma_4$ is redundant as a generator but we need to consider it in order to perform our surgeries equivariantly. We will describe these curves precisely later in \Cref{ssec:ToriGamma}, for a schematic picture see \Cref{Fig:Gammas}(Right).

 \begin{figure}
     \centering
     \includegraphics[scale=0.25]{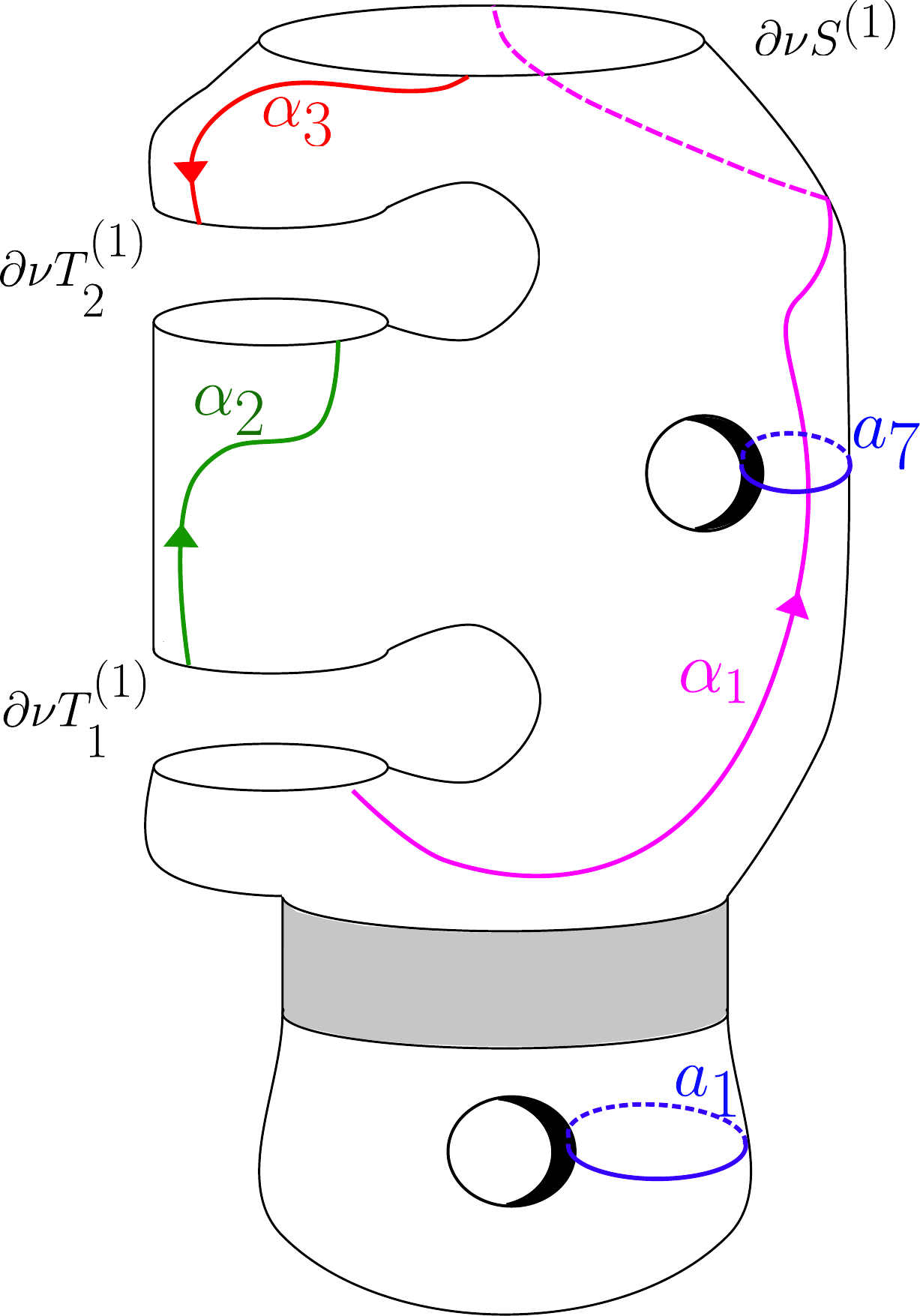}\hspace{1cm}
     \raisebox{1.5cm}{\includegraphics[scale=0.35]{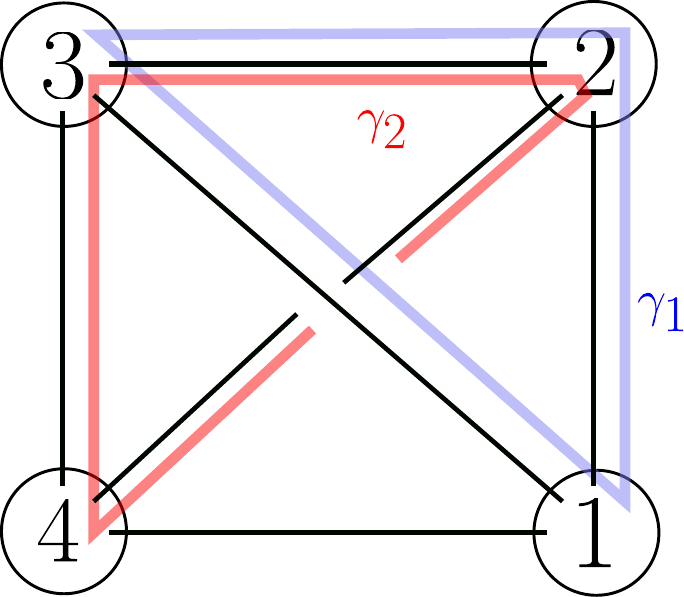}}
    \caption{\label{Fig:Gammas}Left: arcs $\alpha_i = \gamma_{2-i} \cap B^{(1)}$ lying on the  fiber of $B^{(1)}\setminus \nu^\circ (S^{(1)}\cup T_1^{(1)}\cup T_2^{(1)})$.
      Right: Schematic picture of $X$ and the loops $\gamma_1,\gamma_2$. Internal edges represent identification of sections, external edges identification of tori.} 
 \end{figure}

  In \Cref{ssec:ToriGamma} we will define disjoint tori $T_{\gamma_i} \subset \X$ that satisfy the following:
\begin{prop}\label{prop:PropertiesOfToriGamma_i} There exist disjoint tori $T_{\gamma_1}, \dots, T_{\gamma_4}$ in $\X$ with the following properties:
\begin{enumerate}
    \item \label{p_longitude}$\gamma_i\subset T_{\gamma_i}$ is a longitude.
    \item \label{p_selfIntersection}$T_{\gamma_i}$ has zero self-intersection for each $i$.
     \item \label{p_equivariant} $T_{\gamma_{i+1}} = \mathbf{1}\cdot T_{\gamma_i} $ for each $i$. 
    \item \label{p_dual} Each $T_{\gamma_i}$ has a geometric dual torus $T_{\gamma_i}^*$ such that $T_{\gamma_i}^*\cap T_{\gamma_j}=\emptyset$ for $i\neq j$.
    \item \label{p_generatorNullHmtpic} A generator of $\pi_1(T_{\gamma_i}^*)$ is null-homotopic in $\X \setminus \bigsqcup_{i=1}^4 \mathrm{int}(\nu T_{\gamma_i}))$.
    \item \label{p_Lagrangian} Each dual torus $T_{\gamma_i}^*$ is Lagrangian, and in particular they have zero self-intersection.
\end{enumerate}
\end{prop}

\subsection{}\label{ssec:ConstructionOfX_p} Consider  $\Z_{4}$-equivariant trivializations of the tubular neighbourhoods $\nu T_{\gamma_i}\simeq \T^2\times \D^2$  and fix $\Z_{4}$-equivariant  basis $\{[\gamma_i], \beta_i, [\partial \D^2]\} \in H_1(\partial (\X \setminus \nu^\circ T_{\gamma_i}))$. 

Let $\theta = (\theta_1,\dots, \theta_4) \in (\Z^2)^4$ be a multindex. We denote by $\X_\theta$  a $4$-manifold   obtained by gluing 
 $\bigsqcup_{i=1}^4 \T^2\times \D^2$ to $\X \setminus \bigsqcup_{i=1}^4 \nu^\circ T_{\gamma_i} $ using as gluing diffeomorphism between the   $i$-th boundary components  $i =1,\dots, 4$  a map 
$\partial (\T^2\times \D^2)\to \partial (\X \setminus \nu^\circ T_{\gamma_i})$
sending 
\begin{equation}
    [\partial \D^2]\in H_1(\partial (\T^2\times \D^2))\mapsto \theta_{i,1}[\gamma_i]+\theta_{i,2}[\partial \D^2] \in H_1\left(\partial(\X \setminus \nu^\circ T_{\gamma_i})\right).
\end{equation}
There might be more diffeomorphisms satisfying this property but they define diffeomorphic $4$-manifolds.

For $p\in \N$, we define $\X_{p}$ as $$\X_{p} := \X_{\left((1,p),(1,p),(1,p),(1,p)  \right)},$$
where we use the same diffeomorphisms on each boundary component.

\begin{lem}\label{lem:TopPropertiesXp}
    For any $p\in \Z$, $\X_{p}$ is simply-connected and carries a free $\Z_{4}$-action. Moreover the manifolds $\{\X_p\}_{p\in \Z}$ are homeomorphic to each other.
\end{lem}
\begin{proof}
    Remember that the loops $\gamma_i$ generate $\pi_1(\X)$. So, it is enough to show that the surgeries yielding $\X_p$ kill these loops and do not introduce new generators in the fundamental group. By  \Cref{prop:PropertiesOfToriGamma_i}, each torus $T_{\gamma_i}$ admits a dual surface $T_{\gamma_i}^*$, and a generator of $\pi_1(T_{\gamma_i}^*)$ is null-homotopic in $\X \setminus \bigsqcup_{i=1}^4 \nu^\circ T_{\gamma_i}$. Thus, we can apply Lemmas~\ref{lem:MeridianTriviality} and \ref{lem:SurgeryKillingLongitude} iteratively to conclude that 
    $$\pi_1(\X_p)\cong \pi_1(\X)/\langle \gamma_1, \dots, \gamma_4 \rangle \cong 1.$$

    Since we are performing the surgeries equivariantly, the free $\Z_{4}$ action on $\X$ extends to a free action on $\X_p$.\\
    Lastly, invoking the classification of simply connected topological $4$-manifolds \cite{Freedman82}, we conclude that all the manifolds $\{\X_p\}$ are homeomorphic to each other: in fact, they are non-spin (see Section \ref{sub:BaykurHamadaManifolds}), and have the same signature and Euler characteristic because these invariants are unchanged by the surgeries.
\end{proof}

\begin{lem}\label{lem:ExoticZ4} The family $\{\X_p\}_{p\in \Z}$  contains infinitely many 
exotic copies of $\X_1$.
\end{lem}
\begin{proof} By \Cref{lem:TopPropertiesXp}  all the manifolds in the family are homeomorphic.
We will show that we can distinguish the diffeomorphism type of an infinite subfamily by looking at their Seiberg-Witten invariants.

A formula for the Seiberg-Witten invariants of a manifold obtained by a surgery along a torus is given in \cite{Morgan1997ProductFormulas}.
Applying \cite[Thm 1.1]{Morgan1997ProductFormulas} once we obtain

$$\sum_i SW(\X_p, k_p +iT_{\gamma_1}) =\sum_i SW(\X_{e_1 p}, k_{e_1 p} +iT_{\gamma_1}) + p  \sum_i SW(\X_{e_2 p}, k_{e_2 p} +iT_{\gamma_1}),$$

where $e_1 = (1,0), e_2 = (0,1)$ and $e_1 p = (e_1, (1,p),(1,p),(1,p))$ and $e_2 p = (e_2, (1,p),(1,p),(1,p))$, and the classes $k_{e_ip}$ are characteristic elements depending on the multindex. 
Another application of \cite[Thm 1.1]{Morgan1997ProductFormulas} gives

\begin{equation*}
    \begin{split}
    \sum_{i_1,i_2} SW(\X_p, k_p & +\sum_{j=1,2}i_j   T_{\gamma_j})  =   \sum_{i_1,i_2} SW(\X_{e_1 e_1 p}, k_{e_1 e_1 p} 
     +\sum_{j=1,2}i_j T_{\gamma_j})\\ & + p\sum_{i_1,i_2}SW(\X_{e_1 e_2 p}, k_{e_1 e_2 p} +\sum_{j=1,2}i_j T_{\gamma_j}) \\  &  + p\sum_{i_1,i_2}SW(\X_{e_2 e_1 p}, k_{e_2 e_1 p} +\sum_{j=1,2}i_j T_{\gamma_j})
     + p^2\sum_{i_1,i_2}SW(\X_{e_2 e_2 p}, k_{e_2 e_2 p} +\sum_{j=1,2}i_j T_{\gamma_j}).\\
    \end{split}
\end{equation*}
where $e_i e_j p = (e_i, e_j, ,(1,p),(1,p))$. 
It is now clear that applying iteratively the formula \cite[Thm 1.1]{Morgan1997ProductFormulas} we obtain that
\begin{equation}\label{eq:SWPolynomial}
    \sum_{i_1,i_2,i_3,i_4} SW(\X_p, k_p +\sum_j i_j T_{\gamma_j}) = p^4 \sum_{i_1,i_2,i_3,i_4}SW(\X, k + \sum_j i_j T_{\gamma_j}) + q(p),
\end{equation}
where $q(p)$ is a polynomial of degree at most $3$ and we used the identification  $\X_{(e_2, e_2, e_2, e_2)}\simeq \X$.

Let $k$ be equal to the canonical class of the symplectic manifold $\X$ (see subsection \ref{subsec:ConstructionXhi}). We claim that $SW(\X, k + \sum_j i_j T_{\gamma_j}) \neq 0$ if and only if 
$i_1=i_2=i_3=i_4 = 0$. By \cite{taubes1994seiberg},
$SW(\X, k )\neq 0$. The dual tori $T_{\gamma_i}^*$
are essential in homology and Lagrangian  \ref{prop:PropertiesOfToriGamma_i}, thus adjunction inequality \cite{kronheimer1994genus,morgan1996product,ozsvath2000symplectic}  gives that $k\cdot T_{\gamma_i}^* = 0$.

If a class $k + \sum_j i_j T_{\gamma_j}$ is basic, by adjunction inequality we obtain that for any $m=1,\dots, 4$,
$$ 0 = -\chi(T_{\gamma_m}^*) \geq \underbrace{T_{\gamma_m}^*\cdot T_{\gamma_m}^*}_{= 0} + | (k+ \sum_j i_j T_{\gamma_j})\cdot T^*_{\gamma_m} | = |i_{m}|,$$ which proves the claim.

Now \eqref{eq:SWPolynomial} and the claim give, for $k$ being the canonical class,
\begin{equation*}
    \sum_{i_1,i_2,i_3,i_4} SW(\X_p, k_p +\sum_j i_j T_{\gamma_j}) = p^4 SW(\X, k) + q(p).
\end{equation*}
On the right hand side we have a non-constant polynomial in $p$. Consequently, denoting by $n_p\in \N$ the number of basic classes of $\X_p$, it follows that 
$$\left(\max_{\mathfrak{s}\in \Spinc} |SW(\X_p, \mathfrak{s})|, n_p\right)\in \N\times \N$$
diverges as $p\to \infty$. 
\end{proof}

\subsection{The tori $T_{\gamma_i}$}\label{ssec:ToriGamma} In this subsection we prove \Cref{prop:PropertiesOfToriGamma_i}. We begin by describing more precisely the curves $\gamma_i\subset \X$, $i=1,\dots, 4$.
Such curves can be described in $X$ because they do not intersect the surgery tori used to produce $\X$ from $X$.
Recall that $X$ is obtained by gluing together the four blocks 
$$B^\ii \setminus \nu^\circ (S^\ii\cup T_1^\ii\cup T_2^\ii)\simeq   B\setminus \nu ^\circ(S\cup T_1\cup T_2), \quad \text{for $i=1, \dots, 4.$}$$

Since $\gamma_{i+1} = \mathbf {1} \cdot \gamma_i$, it is enough to describe
the intersection of the curves $\gamma_i$ with the first block $B^{(1)} \setminus \nu^\circ (S^{(1)}\cup T_1^{(1)}\cup T_2^{(1)})$ thus we set $$\alpha_i := \underbrace{(\mathbf{1-i})}_{\in \Z_4}\cdot(\gamma_1\cap B^{(i)})=\gamma_{2-i} \cap B^{(1)} \quad \text{ for $i=1,\dots, 4$}.$$
The arcs $\{\alpha_i\}_i$  lie on the fiber over a point  $p\in x_1\subset \Sigma_2$. More precisely, let $F\subset B^{(1)}$ be the fiber of $B^{(1)}\to \Sigma_2$ over $p$ and let $F^\circ = F \setminus \nu^\circ (S^{(1)}\cup T_1^{(1)}\cup T_2^{(1)})$.
Then the arcs $\alpha_1,\alpha_2, \alpha_3$ are contained in $ F^\circ$ and are given by  \Cref{Fig:Gammas}(this is how we define them).
Notice that $\gamma_4 \cap B^{(1)} = \emptyset$ hence there is no arc $\alpha_4$.

\begin{remark}
    The endpoints of $\alpha_1$ and $\alpha_3$ intersect $\partial \nu S^{(1)}$ in two antipodal points. This is because when we defined the gluing map \eqref{eq:GluingMapSection1}
    we mapped the normal circle to itself as $z\mapsto -z$.
    
    On the other hand, this does not happen for the intersections with $\partial \nu T_1^{(1)}$ (and $\partial  \nu T_2^{(2)}$). 
    Indeed $\nu T_1 \simeq (x_1\times \D^1)\times (\alpha_6 \times \D^1)$
    and we can suppose that the gluing map acts on the normal circle as 
    $$\partial(\D^1\times \D^1) \to \partial (\D^1\times \D^1)$$
    $$(t_1,t_2)\mapsto (-t_1, t_2).$$
    This explains why $\alpha_1$ and $\alpha_2$ can intersect $\partial \nu T_1^{(1)}$ in points with the same $\D^1$-coordinate.
\end{remark}

Now that we have a precise definition of the loops $\gamma_i$, we can define 
\begin{align}\label{eq:DefTGammaAndDuals}
T_{\gamma_i} = x_1\times \gamma_i \subset X &  & T_{\gamma_i}^* =  y_1 \times a_7 \subset B^{(i)}\subset X.  
\end{align}
More precisely, $\gamma_i$ intersects the $j$th- block in an arc $\gamma_i\cap B^{(j)}$ and on each block we can consider a loop in the base of the fibration given by $x_1\subset \Sigma_2$.
The torus $T_{\gamma_i}$ is obtained by gluing together 
the cylinders $x_1\times (\gamma_i\cap B^{(j)})$ for $j=1,\dots, 4$ . Recall that $x_1$ is fixed by $r:S\to S$ by assumption.

The dual torus $T_{\gamma_i}^*$ instead lies entirely in the block $B^{(i)}$ and is defined using again the fibration $B^{(i)}\to \Sigma_2$.

Once again, since the tori $T_{\gamma_i}$ and their duals do not intersect the surgery tori in $X$ used to produce $\X$, they define tori in $\X$.

\begin{proof}[Proof of \Cref{prop:PropertiesOfToriGamma_i}]
\ref{p_longitude}), \ref{p_equivariant}), \ref{p_dual} are straightforward consequences of our construction.

We prove \ref{p_selfIntersection}), by equivariancy it is enough to  compute the self-intersection of $T_{\gamma_1}$. Consider a pushoff $\tilde \gamma_1$ of $\gamma_1$ using the normal lying in $F^\circ$.
From the picture we see that 
$\tilde \gamma_1\cap \gamma_1 \neq \emptyset$. 
Hence, the torus $x_1\times \tilde\gamma_1$, intersects $T_{\gamma_1}$ along  the circles $x_1\times (\tilde \gamma_1\cap \gamma_1)$.
This intersections however are not transversal and indeed it is not difficult to see, working locally to $\tilde \gamma_1\cap \gamma_1$,  that can be perturbed away resulting in a pushoff of $T_{\gamma_1}$ which does not meet $T_{\gamma_1}$.

Item \ref{p_generatorNullHmtpic}) follows from the fact that the loop $y_2$ is null-homotopic, indeed the surgieries yielding $\BaykurHamadaSurgered$ from $B$ kill all the generators of $\pi_1(B)$ coming from the base of the fibration. 

\ref{p_Lagrangian}) 
The tori $\{T_{\gamma_i}^*\}_i$ 
are disjoint from all the surfaces used to glue the copies of $B$ together and to perform the surgeries yielding $\X$.
Therefore near $\{T_{\gamma_i}^*\}_i$
the symplectic structure
is the same as in a copy of the Lefschetz fibration $B$.  Since, for all $i$, $T_{\gamma_i}^*$ is a product of a loop in the base and a loop in the fiber of $B$, $T_{\gamma_i}^*$ is Lagrangian.
\end{proof}

\subsection{Proof of \Cref{Thm1} case $k=1$} We are now ready to prove the first case of \Cref{Thm1} i.e. when the fundamental group is $\Z_4$.
Recall that by \cite[Thm. C]{HambletonCancellation}  there is exactly one homeomorphism class of closed, oriented, smoothable  topological $4$-manifolds with finite cyclic fundamental group and non-spin universal cover for any signature and Euler characteristic.

\begin{proof} The cases where $\min (b_2^+, b^-_2)$ is odd follow by applying \cite[Lem. 13]{Torres2} to the Baykur-Hamada manifold $\BaykurHamadaSurgered$, jointly with blow-ups and orientation reversal. We only need to cover the cases where $\min (b^+_2, b^-_2)$ is even.

    Consider the family $\{\X_p/\Z_{4}\}_{p\in S}$ where $S$ indexes the subfamily of \Cref{lem:ExoticZ4}. The manifolds in this family are pairwise non-diffeomorphic because their universal covers are non-diffeomorphic. By  \ref{item:Non_SpinBaykurHamada}) of \Cref{sub:BaykurHamadaManifolds}, their universal cover is non-spin  and have the same signature and Euler characteristic. 
    Hence,  by \cite[Thm. C]{HambletonCancellation}, they are all homeomorphic to each other. \\
    We can compute their topological invariants explicitly: notice that removing the genus $2$ surface $S$ from $B$ increases the Euler characteristic by two, while removing tori, surgery along tori, and gluing along  $3$-manifolds leave the Euler characteristic unchanged. Thus, $\chi(\X_p)= 4 (\chi (B)+2).$
    Since the signature is a cobordism invariant and both  signature and Euler characteristic are multiplicative under coverings, we deduce that 
    $$\sigma (\X_p/\Z_{4})=0, \ \ \chi(\X_p/\Z_{4})= \chi(B)+2= 34,$$

    where we used the fact that $\BaykurHamadaSurgered$ is homeomorphic to $\#_{15}(\CP^2\#\overline{\CP}^2)$.   
    Thus, since $b_1=0,$
$$b_2^+(\X_p/\Z_{4})=b_2^-(\X_p/\Z_{4})=16.$$

Increasing the genus of the fiber of $B$ from $g=8$ to $g=8+h$, $h>0$, increases $\chi(B)$ by $4h$.
Hence, repeating our construction,  we obtain examples with $b^+_2$ increased by $2h$ and signature zero.
Lastly, we can blow up (equivariantly) and change the orientation of our manifolds to obtain examples with $b_2^+\neq b_2^-$.\end{proof}

\section{Fundamental group $\Z_{4k}$}\label{sec:GroupZ4k}

We  now generalize the previous construction to cover all fundamental groups $\Z_{4k}$ for $k\geq 1$.\\

\subsection{The construction.} We maintain the notation from  the previous section. Consider $4k$ copies of the Baykur-Hamada fibration $\{\BaykurHamada^{(i)}\}_{i=1}^{4k}$. Let $X$ be the manifold obtained by  removing from $B^{(i)}$ tubular neighborhoods of the section $S^{(i)}$ and of the tori $T_1^{(i)}, T_2^{(i)}$  for each $i$, and then gluing them along the boundaries
with maps 
\begin{align*}
   &  \partial \nu S^{(i)} \rightarrow \partial \nu S^{(i+2k)}, & &  \partial \nu T_1^{(i)} \rightarrow \partial \nu T_2^{(i+1)} \\
   & (x, z)\mapsto (r(x),-z) & & (x,z)\mapsto (x,\overline{z}).
\end{align*}

Arguing as before, the manifold $X$ is symplectic and $\Z_{4k}$ acts freely on it sending the $i$th-block to the $(i+1)$th-block. We then perform the same Luttinger surgeries as done by Baykur-Hamada in each block  and call $\X$ the result. Now, $\pi_1(\X)$ is generated by loop $\gamma_1, \dots, \gamma_{4k}$, where $\gamma_1$ goes across the blocks $B^{(1)}, B^{(2)}, \dots, B^{(2k+1)}$, and $\gamma_{i+1}= \textbf{1} \cdot\gamma_i$ see \Cref{Fig:Gammas4k}(Right).
As in \Cref{ssec:ToriGamma}, we describe such curves by their intersections $\alpha_i := (\mathbf{1-i})\cdot (\gamma_1\cap B^{(i)})= \gamma_{2-i} \cap \BaykurHamada^{(1)}$, shown in Figure~ \ref{Fig:Gammas4k}. 
Note that in \Cref{Fig:Gammas4k}  the discs $D_1,D_2$ are removed and their boundaries identified, the  loop $a_7$ appears as $\partial D_2$.
This allows the arc $\alpha_2$ to not cross the other arcs $\alpha_i$'s.

\begin{figure}
    \centering
    \includegraphics[scale=0.3]{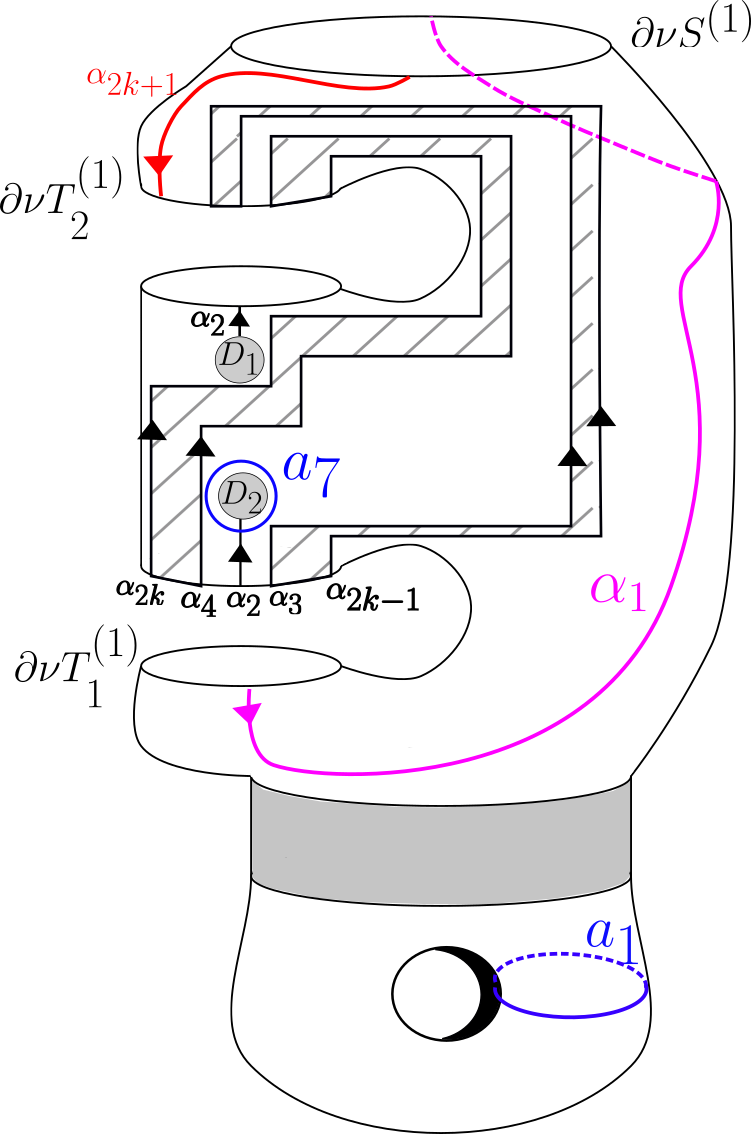}
   \includegraphics[scale=0.4]{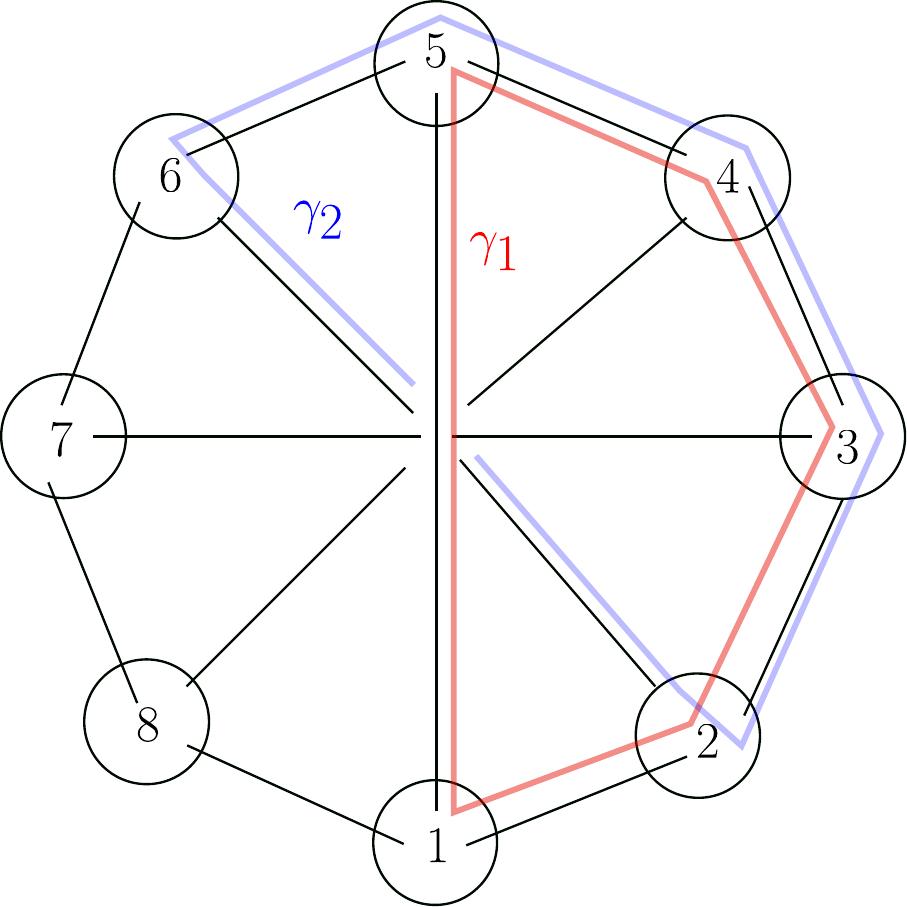}
   \caption{\label{Fig:Gammas4k} Left: the arcs $\alpha_i = \gamma_{2-i} \cap B^{(1)}$ lying on the  fiber of $B^{(1)}\setminus \nu (S^{(1)}\cup T_1^{(1)}\cup T_2^{(1)})$. Note that $a_7 = \partial D_2$ (this is $a_7$ not $\alpha_7$). Right: schematic depiction of $X$ for $k=2$. Each vertex is a block $B^{(i)}$, the internal edges represent identification of sections whilst the external edges are due to the identification of tori. The picture also shows how the paths $\gamma_1$ and $\gamma_2$ cross the blocks.} 
\end{figure}

For $i=1,\dots, 4k$, we define the tori $T_{\gamma_i}\subset \X$ and their duals $T_{\gamma_i}^* $ as 
\begin{align}
T_{\gamma_i} = x_1\times \gamma_i \subset X &  & T_{\gamma_{i}}^* =  y_1 \times a_7 \subset B^{(i+1)}\subset X.  
\end{align}
These tori satisfy the thesis of \Cref{prop:PropertiesOfToriGamma_i}. Analogously to  \Cref{ssec:ConstructionOfX_p}, we define the manifold  $\X_p$ by performing $\Z_{4k}$ equivariant surgeries along the tori $T_{\gamma_i}$ with parameter $(1,p)$.

\begin{lem}
    For each $p\in \Z$, the manifolds $\X_p$ is simply-connected, homeomorphic to $\X_1$, and admits a free $\Z_{4k}$ action. Moreover, the family $\{\X_p\}_{p\in \Z}$ contains infinitely many exotic manifolds.
\end{lem}
\begin{proof}
    Same as in \Cref{lem:TopPropertiesXp} and \Cref{lem:ExoticZ4} mutatis mutandis.
\end{proof}
\subsection{Proof of \Cref{Thm1}, case $k>1$} We can now prove the general case of \Cref{Thm1}.
\begin{proof} The examples are constructed from  $\{\X_p/\Z_{4k}\}_{p\in \Z}$ as in the proof for the case $k=1$. 
\end{proof}
\section{Fundamental group $\Z_{2}\times G$ with $G$ finite}\label{sec:Z2G}
In this section we prove \Cref{Thm2}, more precisely
we will give a construction to produce exotic pairs with signature zero, even $b^+$ and fundamental group $\Z_{2}\times G$ for any given finite group $G$. 

\newcommand{\BaykurStipSzaboSurgered}{\mathcal{L}}

\subsection{The fibration $L$ and the manifold $\BaykurStipSzaboSurgered$.}
First of all we enlarge the base of the Baykur-Hamada fibration by fiber summing with the trivial fibration over $\Sigma_2$ obtaining $B'\to \Sigma_4$ with a section $S$ with $S\cdot S=0$.

Consider the free, orientation reversing, involution  $r:\Sigma_4\to \Sigma_4$ described in \Cref{Pic:InvolutionSigma4}. 
\begin{figure}
    \centering
    \includegraphics[scale=0.5]{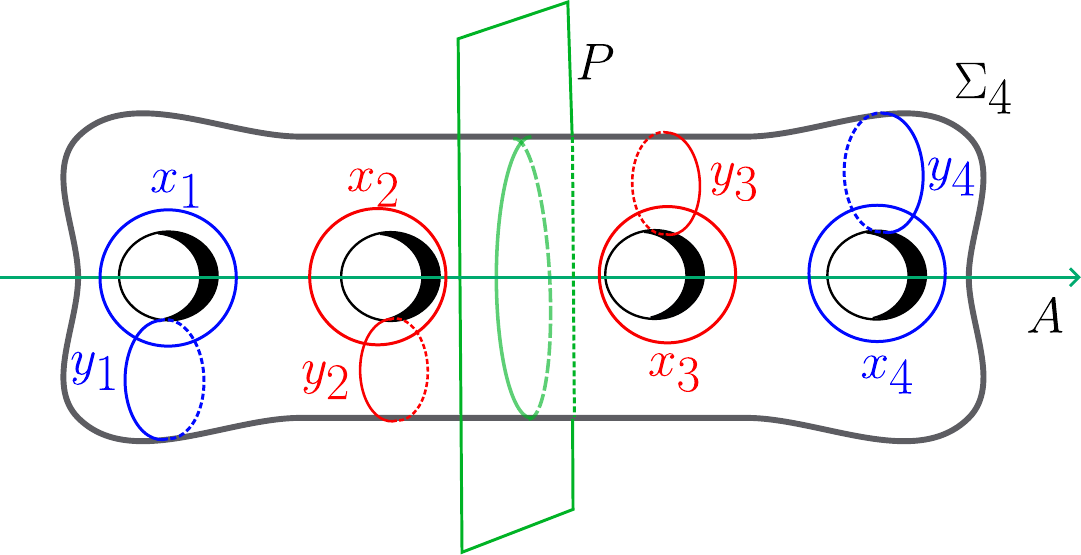}
    \caption{\label{Pic:InvolutionSigma4} Surface $\Sigma_4$ with the curves $x_i, y_i$, $i=1,\dots,4$. The involution $r$ is given by a $\pi$-rotation in the axis $A$ followed by a reflection in the plane $P$.}
\end{figure}

We define a Lefschetz fibration $L\to \Sigma_4$  by 
gluing together $B'\setminus \nu S$
to  itself with the map
$$\Sigma_4\times \S^1\to \Sigma_4\times \S^1 = (x,z)\mapsto (r(x), z).$$

Notice that the fiber of $L$ at $p\in \Sigma_4$ is the connected sum of the fiber of $B'$ at $p$ with the fiber at $r(p)$.
In particular $L$ has twice the genus of $B'$ and twice its singularities.

$L$ carries a free $\Z_{2}$-action given by $r$ on the base and a reflection $\tau$ on the fiber fixing the circle splitting the connected sum.
We will denote by  $F$  the generic fiber of $L$  and by $\mu$ a product of right handed Dehn twists and two commutators in $Aut(F)$ defining the fibration. We can assume  that the monodromy defined by $\mu$ fixes a disk $D\subset F$ since $L$ has a section with zero self-intersection. 
The involution $\tau:F\to F$ preserves $D$, and acts on it as a reflection.

Performing  $\Z_{2}$-equivariantly on $L$ the surgeries that yield $\BaykurHamadaSurgered$ from $B$,  we obtain a simply-connected symplectic manifold $\BaykurStipSzaboSurgered$ with a free $\Z_{2}$-action.

This construction is originally due to Baykur-Stipsicz-\Szabo who performed it in \cite[section 3]{BaykurStipsiczSzabo}, except for enlarging the base of $B$.  This will come in handy later to find more disjoint, equivariant, surgery tori.

\begin{remark}\label{rem:ToriKnotSurgery}
There exists in $\BaykurHamadaSurgered \setminus S$ a homologically essential torus $T$ with simply-connected complement ($T_1$ or $T_2$ with the notation of \Cref{sub:BaykurHamadaManifolds}). Hence, there is a pair of homologically essential tori with simply-connected complement on which we can perform $\Z_{2}$-equivariantly knot surgery \cite{Fintushel1998KnotsLinksAnd4Manifolds} to obtain exotic copies of $\BaykurStipSzaboSurgered$.
\end{remark}

\subsection{The manifold $X$}
\subsubsection{Cayley graphs.} Let $G$ be a finite group and choose a finite presentation for $G$:
\begin{equation}\label{eq:presentation}
    G = \langle g_1,\dots, g_e \ | \ R_1,\dots, R_r\rangle
\end{equation}
where $e,r \in \N$ and $\{R_i\}_i$ are some relators.
We assume that, in $G$,  $g_j\neq g_i, g_i^{-1}$ for $i\neq j$. The relators may contain the inverse of the generators.
We denote by $\Gamma = (G,E)$ the  Cayley graph for $G$ induced by \eqref{eq:presentation}. By definition $\Gamma$ has a vertex for each element of $G$ and we create an undirected edge from $x$ to $g_i x$
for any generator $g_i$.
As a result, $\Gamma$ has $ |G|$ vertices and $|E| = e|G|$ edges.

\subsubsection{The surface $F_G$}Let $F_G$ be the oriented surface obtained from $\Gamma$ by replacing each vertex with $F$  and each arc with a tube. 
$F_G$ carries an obvious free, orientation preserving, $G$-action induced by the free action of $G$ on $\Gamma$. We can promote it to a $\Z_{2}\times G$-action. 
$F_G$ can be thought as
a connected sum along the graph $\Gamma$ of $|G|$ disjoint copies of $F$ at the points $p_1,\dots, p_{2e}\subset F$.
Without loss of generality we can assume that $\{p_i\}_{i=1}^{2e} $ is contained in $D$ and fixed by $\tau$. 

In order to define an involution of $F_G$ it is enough to consider an involution of  $F$ fixing $\{p_i\}_{i=1}^{2e}$, we choose $\tau$.
The resulting involution acts in the same way on each vertex and edges so it commutes with the $G$-action.

\subsubsection{The manifold $X$} We want to define a Lefschetz fibration  with fiber $F_G$ and base $\Sigma_4$. Let $F^\circ = F\setminus \mathrm{int}(D)$. For each $g\in G$ there is an  inclusion $\iota_g : F^\circ \to F_G$ inducing an embedding $MCG(F^\circ,\partial F^\circ)\to MCG(F_G, F_G\setminus \iota_g(F^\circ))$.
We can therefore concatenate $|G|$ copies of the factorization $\mu$ 
and interpret it as a product of $4|G|$ commutators in $MCG(F_G)$.
Actually, this is a product of just $4$ commutators because the mapping classes are supported in distinct copies of $F^\circ\subset F_G$ induced by the embeddings $\iota_g$s.
 
 We call the resulting $4$-manifold $X$. The manifold
    $X$ has a free $G$-action acting fibewise as $G$ on $F_G$. 
    Notice that singular fibers do not pose any problem as vanishing circles are sent to vanishing circles by this action.
    
    This action can be promoted to a free $\Z_{2}\times G$-action. 
    On each fiber we already defined an action of $\Z_{2}$, which however is not free, to overcome this problem, we compose this action with the fixed point free, orientation reversing involution $r:\Sigma_4\to \Sigma_4$ given in \Cref{Pic:InvolutionSigma4}.
    The result is a free orientation-preserving $\Z_{2}\times G$-action.
    \begin{remark}\label{rem:r_loops}
        Notice that (see \Cref{Pic:InvolutionSigma4})
    $r$ exchanges the following loops in $\Sigma_4$:
    \begin{align*}
    &x_1\longleftrightarrow x_4 & & y_1\longleftrightarrow y_4\\    
    &x_2\longleftrightarrow x_3 && y_2\longleftrightarrow y_3.\\
    \end{align*}
    \end{remark}

    \begin{lem} $X$ is closed and the  $\Z_{2}\times G$-action defined above is free and orientation preserving.
    The Lefschetz fibration admits a section hence $X$ admits a symplectic structure. Moreover it has zero signature. 
    \end{lem}
    \begin{proof}     
    Since $\mu$ fixes $D\subset F$,  there is a fibration preserving embedding of $L$ minus a section in $X$. Hence $X$ has a section and $[F_G]\neq 0\in H_2(X)$, thus $X$ can be given a symplectic structure \cite[Thm 10.2.18]{gompfStipsicz4ManifoldsAndKirbyCalculus},.

    The signature of a closed Lefschetz fibration depends on the vanishing circles and their framings. 
    In the case of $X$, the latter consist of $|G|$ disjoint collections of the vanishing circles of $L$. Thus the claim follows from $\sigma(L) = 0$.
\end{proof}

\subsubsection{Stabilizations of $X$}\label{sssec:StabilizationsX}Our construction of $X$ carries through even if we modify $F_G$ by adding genus to it $\Z_{2}\times G$-equivariantly. More precisely, this means $\Z_{2}\times G$-equivariantly removing $\epsilon|G|$ copies of  $\S^0\times \mathrm{int}(\D^2)$ from $F_G$ and  gluing in $\epsilon|G|$ copies of  $\D^1\times \S^1$ along the boundary, where $\epsilon = 1$ if $\S^0\times \ast$ is embedded in the fixed point set of $\tau$, and $\epsilon = 2$ otherwise.

We refer to the induced operation on $X$ as \emph{stabilization}. 
\subsection{Killing $\pi_1(X)$} We will kill the fundamental group of $X$ by performing  Luttinger surgeries equivariantly. In order to find a suitable collection of surgery tori we may have to stabilize our original $X$ several times (see \Cref{sssec:StabilizationsX}). With a slight abuse of notation we will still denote this manifold as $X$.

\subsubsection{Step 1.}The fundamental group of $X$ is generated by curves in the fiber $F_G$ and curves in the base $\Sigma_4$. By performing $|G|$-times the same surgeries done in \cite{BaykurStipsiczSzabo}, we kill the generators coming from $\Sigma_4$ and from the copies of $F^\circ\subset F_G$. 
It remains to kill the loops in $F_G$ induced by the the graph $\Gamma$, i.e. $\pi_1(F_G)$ for $F=\S^2$.

\subsubsection{} Denote by $S_G$ the surface $F_G$ obtained for $F = \S^2$ and let $\genus$ be its genus. 
Let 
\begin{equation}\label{def:LoopGammas}\gamma_1,\dots, \gamma_{2\genus}\subset S_G\end{equation}
be a collection of embedded closed loops supporting $\pi_1(S_G)$, i.e. attaching $2$-cells along \eqref{def:LoopGammas} results in a $1$-connected space. We can assume, up to an arbitrarily small $C^\infty$-perturbation, that each curve in \eqref{def:LoopGammas} intersects transversally its $G$-orbit.

\begin{lem}\label{lem:StabilizationDisjoint}
    Up to $\Z_{2}\times G$-equivariant stabilizations of $S_G$ and a modification of the curves \eqref{def:LoopGammas},  we can assume that the $G$-orbit of
     $\gamma_i$
    consists of disjoint loops  for any $i=1,\dots, 2\genus$.
\end{lem}
\begin{proof}
    Suppose that for some  $g\in G$ and $i$, $\gamma_i$ and $g\cdot \gamma_i $ intersect transversally at $p$. Then, locally at $p$, we stabilize $S_G$ by adding genus $\Z_{2}\times G$-equivariantly and modify $\gamma_i$ (and his orbit)
    as in \Cref{Pic:StabilizationIntersections}.
    To take care of the newly added genus, we  add new  loops 
    $$\alpha_1,\beta_1,\dots, \alpha_{N}, \beta_N$$
    to \eqref{def:LoopGammas},  defined by the orbit of the loops shown in \Cref{Pic:StabilizationIntersections}, thus  $N = |G|$ or $N = 2|G|$ depending whether $p$ belongs to the fixed point set of the $\Z_{2}$-action. Notice that the $G$-orbit of the $\alpha_i$s or the $\beta_i$s has no intersections. Now the new collection of loops supports the fundamental group of the stabilized surface and we removed the intersection at $p$ without introducting new intersections. We repeat this operation until the thesis is satisfied.
\end{proof}

\begin{figure}
    \centering
    \includegraphics[width=0.7\linewidth]{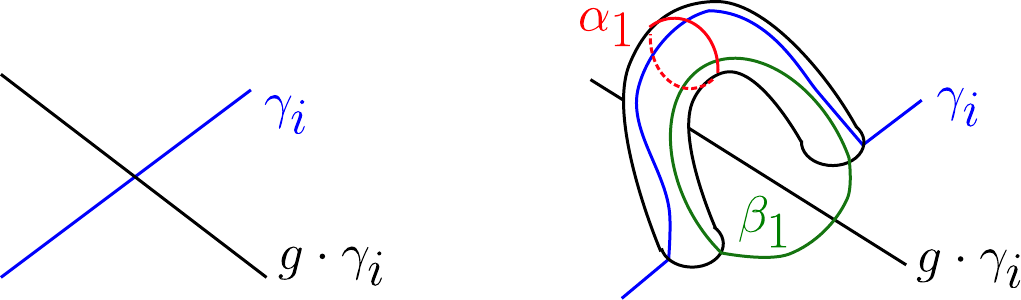}
    \caption{Left: a neighbourhood of $p$ in $S_G$. Right: addition of genus and modification of $\gamma_i$ with consequent elimination of the intersection point. Notice that the whole orbit of $\gamma_i$ is modified by this operation.}
    \label{Pic:StabilizationIntersections}
\end{figure}
In view of \Cref{lem:StabilizationDisjoint}, we can assume that the $G$-orbit of the $\gamma_i$s consists of disjoint loops. Up to reorder \eqref{def:LoopGammas},
  there is $s\leq 2g$, such that $\gamma_1,\dots, \gamma_s$ belong to different $G$-orbits and for any $i>s$, $\gamma_{i} = g_{i,j}\cdot \gamma_j$ for some $j\leq s$, $g_{i,j}\in G$.

  Now for each $i=1,\dots, s$, we choose a point $p_i \in \gamma_i$ not fixed by $\Z_{2}$  and perform a $\Z_{2}\times G$-equivariant stabilization of the surface and modify $\gamma_i$ at $p_i$ as in \Cref{Pic:StabilizationIntersections}. We will still denote the result by $\gamma_i$. Note that this modifies the whole  orbit of $\gamma_i$. We add the $\alpha_i$s and $\beta_i$s to our new collection of curves, which now looks like
 
 \begin{equation}\label{def:LoopsStabilized}
 \gamma_1,\dots, \gamma_{2\genus},\alpha_1,\dots,\alpha_{2|G|s}, \beta_1,\dots, \beta_{2|G|s}    
 \end{equation}
 
 Define 
 \begin{align*}
     & \alpha_i^* = \beta_i & & \beta_i^* = \alpha_i \\
    &  \gamma_i^* = \alpha_i \ \text{ for $i\leq s$} & & \gamma_i^* = (g_{i, j}\gamma_j)^* = g_{i,j}\cdot \alpha_j \ \text{ for $i> s$}
 \end{align*}
 \begin{lem} \label{lem:DualCurves}The collection \eqref{def:LoopsStabilized} supports the fundamental group of the stabilized surface.
 Moreover, among all the $G$-orbit of the curves \eqref{def:LoopsStabilized}, $\gamma_i^*$ intersects only $\gamma_i$ in one point and one of the $\beta_i$s, $\alpha_i^*$ intersects only $\alpha_i$ in one point, $\beta_i^*$ intersects only $\beta_i$ in one point and possibly one of the $\gamma_i$s.
 \end{lem}   

 Thus after stabilizing $X$, we can define the following Lagrangian tori in $X$: 
\begin{align*}
 & T_{\gamma_i} :=  x_1\times \gamma_i & & T_{\alpha_i} := x_1\times \alpha_i  & & T_{\beta_i} := x_2\times \beta_i,
\end{align*}
where $i=1,\dots, s$.
To be more precise, if there are intersections between the tori $T_{\gamma_i}, T_{\alpha_j}$ we take a parallel pushoff of $x_1$ on $\Sigma_4$ to get rid of them.
In this way,  the  tori $(\Z_{2}\times G)\cdot \{T_{\gamma_i}, T_{\alpha_j}, T_{\beta_k}\}_{i,j,k}$ are  disjoint thanks to \cref{rem:r_loops}, and do not meet the surgery tori used in Step~1. 
Moreover, by \Cref{lem:DualCurves}, each torus has a dual torus
\begin{align*}
 & T_{\gamma_i}^* := y_1 \times \gamma_i^* 
 & & 
 T_{\alpha_i}^* := y_1\times \alpha_i^*
 &  & T_{\beta_i}^* := y_2\times \beta_i^*,
\end{align*}
that intersects it once and does not intersect any other torus in  $(\Z_{2}\times G)\cdot \{T_{\gamma_i}, T_{\alpha_j}, T_{\beta_k}\}_{i,j,k}$ nor the tori used for the surgeries in Step~1. 

Here it is key that the loops associated with the $\beta_i$s are $x_2,y_2$ and $x_3,y_3$ (due to the involution) which are disjoint from the loops $x_1,y_1$ and $x_4,y_4$
  associated with the $\gamma_i$s and $\alpha_i$s.
This is precisely the reason why we had to enlarge the base of $B$ to $\Sigma_4$.

\subsubsection{Step 2}
After stabilizing $X$ and performing the surgeries of Step~1, we  perform equivariant Luttinger surgery  on the tori $(\Z_{2}\times G)\cdot \{T_{\gamma_i}, T_{\alpha_j}, T_{\beta_k}\}_{i,j,k}$, using as longitude the subscript curve. We denote by $\X$  the resulting manifold.
\begin{lem}\label{lem:ManifoldXForZ2G}
    $\X$ carries a free $\Z_{2}\times G$ action, $\sigma(\X) = 0$ and $\pi_1(\X) = 1$. In addition, 
    $\X$ is symplectic and irreducible. 
\end{lem}
\begin{proof} Since  the surgeries are done equivariantly, the $\Z_{2}\times G$-action carries over the result of the surgery. 
Since $\sigma(X)=0$ and surgeries preserve the signature, $\sigma(\X)=0$. 

Now we will show that $\pi_1(\X) =1$.
The surgeries in Step~1 kill all the generators of $\pi_1(X)$ except for those coming from $\pi_1(S_G)$.
The $\Z_{2}\times G$-orbit of the longitudes of the tori $\{T_{\gamma_i}, T_{\alpha_j}, T_{\beta_k}\}_{i,j,k}$ contain \eqref{def:LoopsStabilized} which supports $\pi_1(S_G)$.
Thanks to the existence of the dual tori we can now invoke \Cref{lem:SurgeryKillingLongitude} to conclude.
Notice that the curves $y_1, y_2\subset \Sigma_2$
become null-homotopic after Step~1.

To show that $\X$ is symplectic and irreducible we mimic the argument of \cite[Thm. 9]{BaykurHamada}.
Since all the tori are Lagrangian, Luttinger surgery yields a symplectic manifold.  $B$ is a relatively minimal Lefschetz fibration \cite{BaykurHamada}, hence so are $L$ and $X$. Thus  $X$ is minimal because $\mathrm{genus}(\Sigma_4)>0$  \cite{StipsiczRelativelyMinimalMinimal}.
Since $\X$ is obtained via Luttinger surgeries, $\X$ is also minimal by \cite[Lemma 5.6]{HoLiLuttingerSurgery}. Now  $\X$ is minimal, and a simply-connected symplectic manifold  which is minimal is also irreducible \cite{HamiltonKotschick}. 
\end{proof}

We can now prove \Cref{Thm2}
\begin{proof}[Proof of \Cref{Thm2}] Construct the manifold $\X$ of \Cref{lem:ManifoldXForZ2G} associated with $G$.

Let $\mathcal{K} =\{K_n\}_{n\in \N} $ be  a sequence of fibered knots in  $\S^3$ with distinct symmetrized  Alexander polynomial $\{\Delta_{K_n}(t)\}_n$.
Now, for $K\in \mathcal K$,  perform $\Z_{2}\times G$-equivariantly knot surgeries  \cite{Fintushel1998KnotsLinksAnd4Manifolds} on $\X$  using the tori of \Cref{rem:ToriKnotSurgery}, denote the result of the surgeries as $\X_{K}$. The manifolds $\{\X_K\}_{K\in \mathcal K}$ satisfy the conclusions of \Cref{lem:ManifoldXForZ2G}. Nevertheless, they are pairwise non-diffeomorphic, indeed the Seiberg-Witten series of $\X_{K_n}$ is given by $SW(\X_{K_n}) = SW(\X)(\Delta_{K_n}(F_G^2))^{2|G|}$ \cite{Fintushel1998KnotsLinksAnd4Manifolds} and $SW(\X)\neq 0$ because $\X$ is symplectic \cite{taubes1994seiberg}.

Consider the manifolds 
\begin{equation}\label{eq:Examples}
    \{\X_K/ (\Z_{2}\times G)\}_{K\in \mathcal K}.
\end{equation}
These manifolds are  pairwise non-diffeomorphic because are distinguished by their universal cover. On the other hand they have fundamental group $\Z_{2}\times G$, share the same Euler characteristic and have signature zero, because these  are multiplicative under coverings. 
There are only finitely many topological, closed, oriented manifolds with such invariants  \cite[pg.87]{HambletonKreckClassification}. Therefore there is an infinite subset of \eqref{eq:Examples} of the same homeomorphism type, which we will call $Q$.  Next we compute $b^+(Q)$, since $\sigma(Q) = 0$, $$b^+(Q) = (\chi(Q)-2)/2 = \frac{\chi(\X)}{4|G|} - 1,$$
we will  show that $\chi(\X)/4|G|$ is odd.

Since $\X$ is obtained by surgery along tori from $X$, $\chi(\X) = \chi(X)$, and the latter is a Lefschetz fibration with $v_X$ vanishing circles.
Consequently,
$$\frac{\chi(\X)}{4|G|} = \frac{\chi(\Sigma_4)\chi(F_G) + v_X}{4|G|} = \frac{\chi(\Sigma_4)\chi(S_G)}{4|G|} - \frac{\chi(\Sigma_4)2|G|\mathrm{genus}(F)}{4|G|} + \frac {v_X}{4|G|}, $$
where we exploited that $\mathrm{genus}(F_G) = \mathrm{genus}(S_G) + |G|\mathrm{genus}(\mathrm{F}).$
It holds that $v_X = |G|v_L = 2|G|v_B$  where $v_L$ ($v_B$) is the number of vanishing circles of $L$ ($B$). Since $v_B = 4$, $\frac{v_X}{4|G|}$ is even.
Also $\mathrm{genus}(F)$ is divided by $4$ because is twice the genus of the fibration $B$ which is even. Hence,
$$\frac{\chi(\X)}{4|G|} \equiv \frac {\chi(\Sigma_4)\chi(S_G)}{4|G|} =  \frac{-6 \cdot 2(1-\mathrm{genus}(S_G))}{4|G|} \equiv \frac{(1-\mathrm{genus}(S_G))}{|G|}\mod 2,$$

Recall that a stabilization
  of $X$ along points fixed by $\tau$ increases $\mathrm{genus}(S_G)$ by $|G|$ (see \Cref{sssec:StabilizationsX}).
  Hence, after possibly a stabilization of $X$, $b^+(Q)$ is even. The manifolds \eqref{eq:Examples} are irreducible because have finite fundamental group and their universal covers are irreducible. \end{proof}

\nocite{*}
\bibliographystyle{alpha}

\bibliography{ref}
\bigskip

\end{document}